\documentclass[a4paper]{amsart}
\usepackage{amsmath,amssymb,amsthm,enumerate}
\title[Global secondary CR invariants in dimension five]
{Global secondary CR invariants in dimension five}
      
\author{TAIJI MARUGAME}
\date{}

\renewcommand\a{\alpha}
\renewcommand\b{\beta}
\newcommand\g{\gamma}
\renewcommand\d{\delta}

\renewcommand\r{\rho}
\renewcommand\th{\theta}

\newcommand\U{\Upsilon}
\newcommand\pa{\partial}
\newcommand\ol{\overline}
\newcommand{\calE}{\mathcal{E}}

\newtheorem{lem}{Lemma}[section]
\newtheorem{thm}[lem]{Theorem}

\newtheorem{prop}[lem]{Proposition}

\theoremstyle{definition}

\numberwithin{equation}{section}

\makeatletter
\@addtoreset{equation}{section}

\makeatother
\address{Mathematical Analysis Team, RIKEN Center for Advanced Intelligence Project (AIP), 1-4-1 Nihonbashi, Chuo-ku, Tokyo 103-0027, Japan}
\address{Department of Mathematics, Graduate School of Science, Osaka University, 1-1 Machikaneyama-cho Toyonaka Osaka 560-0043, Japan}
\email{taiji.marugame@riken.jp}



\begin{document}

\begin{abstract} 
A global secondary CR invariant is defined as the integral of a pseudo-hermitian invariant which is independent of a choice of pseudo-Einstein contact form. We prove that any global secondary CR invariant on CR five-manifolds is a linear combination of the total $Q'$-curvature, the total $\mathcal{I}'$-curvature, and the integral of a local CR invariant.
\end{abstract}
\maketitle

\section{Introduction}
Construction of invariants of CR structure is a fundamental problem in CR geometry. Of particular importance are global CR invariants which are given as the integral of pseudo-hermitian invariants, i.e., Tanaka--Webster curvature quantities, since one can calculate them by local geometric data of CR manifolds. Recently, there has been much progress in construction of such invariants defined for a restricted class of contact forms, namely pseudo-Einstein contact forms. A {\it global secondary CR invariant} is the integral of a pseudo-hermitian invariant which is defined for any pseudo-Einstein contact form $\theta$ and is invariant under changes of $\theta$. Such a class of invariants is natural in relation to complex geometry and easier to construct by using techniques of the ambient metric or Cheng--Yau's K\"ahler-Einstein filling. 

An example of global secondary CR invariant is the {\it Burns--Epstein invariant} introduced by Burns--Epstein \cite{BE1, BE2} in dimension three, and generalized by the author \cite{Mar1} to higher dimensions as the boundary term of the renormalized Chern--Gauss--Bonnet formula for strictly pseudoconvex domains. Another example is the {\it total $Q'$-curvature} introduced by Case--Yang \cite{CY} in dimension three and generalized to higher dimensions by Hirachi \cite{H}. The $Q'$-curvature is a pseudo-hermitian invariant 
whose transformation formula is described by linear differential operators $P, P'$ and the CR invariance of its integral follows from the formal self-adjointness of these operators. In dimension three, these two CR invariants coincide up to a universal constant multiple; in fact, up to universal constant multiples, the Burns--Epstein invariant is the only global secondary CR invariant of three dimensional CR manifolds; see \cite{H}.

In dimension five, Case--Gover \cite{CG} derived an explicit formula of the $Q'$-curvature by using the CR tractor calculus. Modulo divergence terms, it reads
\begin{equation}\label{Q-prime}
Q'\equiv 2P^3-2P|A_{\a\b}|^2-S_{\a\ol\b\g\ol\mu}A^{\a\g}A^{\ol\b\ol\mu}-|V_{\a\ol\b\g}|^2.
\end{equation}
Moreover, they defined a pseudo-hermitian invariant, called the {\it $\mathcal{I}'$-curvature}, whose integral gives the difference between the total $Q'$-curvature and the Burns--Epstein invariant:
\begin{equation}\label{I-prime}
\mathcal{I}'=\frac{1}{2}P|S_{\a\ol\b\g\ol\mu}|^2+|V_{\a\ol\b\g}|^2+\frac{1}{8}\Delta_b |S_{\a\ol\b\g\ol\mu}|^2.
\end{equation}
The transformation formula of $\mathcal{I}'$ is described in terms of a CR invariant representative of the second Chern class of the CR tangent bundle, whose vanishing explains the CR invariance of the total $\mathcal{I}'$-curvature.

Hirachi \cite{H} conjectured that a pseudo-hermitian invariant defined for pseudo-Einstein contact forms is decomposed into the sum of a multiple of $Q'$, local CR invariants, and divergence terms if it integrates to a global secondary CR invariant. However, Reiter--Son \cite{RS} showed that the total $\mathcal{I}'$-curvature gives a counterexample to Hirachi's conjecture. The $\mathcal{I}'$-curvature was generalized to higher dimensions independently by the author \cite{Mar2} and Case--Takeuchi \cite{CT}. Case--Takeuchi also proved that the $\mathcal{I}'$-curvatures give counterexamples to Hirachi's conjecture in general dimensions. 

In this paper, we prove that the $\mathcal{I}'$-curvature is essentially the only counterexample to Hirachi's conjecture in dimension five: 

\begin{thm}\label{main-thm}
Let $F_\theta\in\calE(-3, -3)$ be a pseudo-hermitian invariant on CR five-manifolds defined for pseudo-Einstein contact forms. If the integral of $F_\theta$ is independent of the choice of pseudo-Einstein contact form, then there exist constants $c_1, c_2, c_3\in\mathbb{C}$ such that
\[
F_\theta\equiv c_1Q'+c_2\mathcal{I}'+c_3S_\a{}^\b{}_\g{}^\mu S_\b{}^\nu{}_\mu{}^\tau S_\nu{}^\a{}_\tau{}^\g
\]
modulo divergence terms. 
\end{thm}
Note that the last term in the above expression is a nontrivial local CR invariant; see \S\ref{reinhardt}. The integrals of each term are linearly independent by the argument in \cite{RS}. We prove this theorem by listing up all pseudo-hermitian invariants $F_\theta\in\calE(-3, -3)$ for pseudo-Einstein contact forms up to divergence terms (see Proposition \ref{list}). Considerations of specific examples of CR manifolds (the sphere and a Reinhardt hypersurface) give the restriction to $F_\theta$ as above if it integrates to a global secondary CR invariant.

\bigskip\noindent {\bf Acknowledgment} The author is grateful to Kengo Hirachi for posing the problem of classifying global secondary CR invariants in dimension five. He also thanks Yoshihiko Matsumoto and Yuya Takeuchi for comments and discussions.

\section{Preliminaries}
\subsection{Pseudo-hermitian geometry of CR manifolds}
Let $(M, H, J)$ be a strictly pseudoconvex CR manifold of dimension $2n+1\ge 5$. Namely, $H\subset TM$ is a contact distribution and $J\in\Gamma(\mbox{End}\ H)$ is an almost complex structure with the integrability condition $[\Gamma(T^{1, 0}M), \Gamma(T^{1, 0}M)]\subset \Gamma(T^{1, 0}M)$. Here $T^{1,0}M\subset\mathbb{C}\otimes H$ denotes the eigenspace bundle for $J$ with eigenvalue $i$. The strict pseudoconvexity means that the Levi form 
$h_\theta(X, Y):=d\theta(X, JY)$ gives a positive definite hermitian form on $H$ for any (positive) contact form $\theta$. 

We fix a contact form $\theta$ and take a local frame $\{Z_0:=T, Z_\a, Z_{\ol\a}:=\ol{Z_\a}\}$ for $\mathbb{C}\otimes TM$ adapted to the decomposition
\[
\mathbb{C}\otimes TM=\mathbb{C}T\oplus T^{1,0}M\oplus \ol{T^{1, 0}M}, 
\]
where $T$ is the Reeb vector field of $\theta$: $\theta(T)=1,\ T\lrcorner\, d\theta=0$. Such a frame is called an {\it admissible frame}. The dual coframe is denoted by $\{\theta^0=\theta, \theta^\a, \theta^{\ol\a}\}$. In this frame, we have 
\[
d\theta=ih_{\a\ol\b}\theta^\a\wedge\theta^{\ol\b},
\]
where $h_{\a\ol\b}:=h_\theta(Z_\a, Z_{\ol\b})$.

The {\it Tanaka--Webster connection} $\nabla$ associated to $\theta$ is a linear connection on $TM$ defined by $\nabla T=0,\ \nabla Z_\a=\omega_\a{}^\b\otimes Z_\b$, where the connection form $\omega_\a{}^\b$ is uniquely characterized by the structure equations
\begin{align*}
d\theta^\a&=\theta^\b\wedge\omega_\b{}^\a+A^\a{}_{\ol\b}\,\theta\wedge\theta^{\ol\b}, \\
dh_{\a\ol\b}&=\omega_\a{}^\g h_{\g\ol\b}+h_{\a\ol\g}\omega_{\ol\b}{}^{\ol\g},
\end{align*}
where $\omega_{\ol\b}{}^{\ol\g}=\ol{\omega_{\b}{}^{\g}}$. Note that the second equation means that $\nabla$ preserves the Levi form, and hence the covariant differentiation commutes with lowering or raising the indices by the Levi form $h_{\a\ol\b}$ or its inverse $h^{\a\ol\b}$. The tensor $A^\a{}_{\ol\b}$ is a component of the torsion tensor of $\nabla$ and is called the {\it Tanaka--Webster torsion tensor}. The tensor $A_{\a\b}:=\ol{A_{\ol\a\ol\b}}$ is symmetric: 
$A_{\a\b}=A_{\b\a}$. 
The curvature form $\Omega_\a{}^\b=d\omega_\a{}^\b-\omega_\a{}^\g\wedge\omega_\g{}^\b$ can be written as
\begin{align*}
R_\a{}^\b{}_{\g\ol\mu}\theta^\g\wedge\theta^{\ol\mu}+\nabla^\b A_{\a\g}\theta^\g\wedge\theta-\nabla_\a A^\b{}_{\ol\g}\theta^{\ol\g}\wedge\theta 
-iA_{\a\g}\theta^\g\wedge\theta^\b+ih_{\a\ol\g}A^\b{}_{\ol\mu}\theta^{\ol\g}\wedge\theta^{\ol\mu}.
\end{align*}
The tensor $R_\a{}^\b{}_{\g\ol\mu}$ is called the {\it Tanaka--Webster curvature tensor} and it has the following symmetries:
\begin{equation}\label{symm-R}
R_{\a\ol\b\g\ol\mu}=R_{\g\ol\b\a\ol\mu}=R_{\a\ol\mu\g\ol\b}, \quad 
\ol{R_{\a\ol\b\g\ol\mu}}=R_{\b\ol\a\mu\ol\g}.
\end{equation}
We also define the {\it Tanaka--Webster Ricci tensor} and the {\it Tanaka--Webster scalar curvature} by the contractions:
\[
R_{\a\ol\b}:=R_\g{}^\g{}_{\a\ol\b}, \quad R:=R_{\a}{}^\a{}_\b{}^\b. 
\]

A {\it (scalar) pseudo-hermitian invariant} is a function defined by a universal formula as a linear combination of complete contractions of the Tanaka--Webster curvature tensor, the Tanaka--Webster torsion tensor, and their covariant derivatives. Note that pseudo-hermitian invariants are relative to a choice of contact form. We usually put CR weights on these invariants and regard them as a CR density; see \S\ref{weights} below.

We will introduce some important tensors on CR manifolds. The {\it Chern--Moser tensor} is defined by
\[
S_{\a\ol\b\g\ol\mu}:=R_{\a\ol\b\g\ol\mu}-P_{\a\ol\b}h_{\g\ol\mu}
-P_{\g\ol\b}h_{\a\ol\mu}-P_{\g\ol\mu}h_{\a\ol\b}-P_{\a\ol\mu}h_{\g\ol\b},
\]
where
\[
P_{\a\ol\b}:=\frac{1}{n+2}\Bigl(R_{\a\ol\b}-\frac{R}{2(n+1)}h_{\a\ol\b}\Bigr)
\]
is the {\it CR Schouten tensor}. The Chern--Moser tensor is trace-free and has the same symmetries as in \eqref{symm-R}. Moreover, it is CR invariant in the sense that $\widehat S_\a{}^\b{}_{\g\ol\mu}=S_\a{}^\b{}_{\g\ol\mu}$ holds for any rescaling of the contact form. We set $P:=P_\a{}^\a=\frac{1}{2(n+1)}R$ and define
\begin{align*}
T_\a&:=\frac{1}{n+2}(\nabla_\a P-i\nabla^\b A_{\a\b}), \\
V_{\a\ol\b\g}&:=\nabla_{\ol\b}A_{\a\g}+i\nabla_\g P_{\a\ol\b}-iT_\g h_{\a\ol\b}-2i
T_\a h_{\g\ol\b}.
\end{align*}
The tensor $V_{\a\ol\b\g}$ is related to $S_{\a\ol\b\g\ol\mu}$ by the following equation 
({\cite[Lemma 2.2]{CT}}):
\begin{equation}\label{V-S}
\nabla_{[\a}S_{\b]}{}^\g{}_\mu{}^\nu=i V_\mu{}^\nu{}_{[\a}\d_{\b]}{}^\g+
i V_\mu{}^\g{}_{[\a}\d_{\b]}{}^\nu.
\end{equation}
In particular, we have
\begin{equation}\label{div-S}
\nabla^{\ol\mu}S_{\a\ol\b\g\ol\mu}=-inV_{\a\ol\b\g}.
\end{equation}
Thus, $V_{\a\ol\b\g}$ is trace-free and satisfies $V_{\a\ol\b\g}=V_{\g\ol\b\a}$.

\subsection{Pseudo-Einstein contact forms}
A contact form is called {\it pseudo-Einstein} if it satisfies
\[
R_{\a\ol\b}=\frac{R}{n}h_{\a\ol\b}.
\]
A complex function $f$ on $M$ is called a {\it CR function} if $\ol Zf=0$ for all 
$Z\in T^{1, 0}M$, and a real function is called a {\it CR pluriharmonic function} when it is locally the real part of a CR function.  If $\theta$ is a pseudo-Einstein contact form, then $e^\U\theta$ is pseudo-Einstein if and only if $\U$ is CR pluriharmonic. 

For a pseudo-Einstein contact form, the Chern--Moser tensor is represented as
\begin{equation}\label{S-pE}
S_{\a\ol\b\g\ol\mu}=R_{\a\ol\b\g\ol\mu}-\frac{2}{n}P(h_{\a\ol\b}h_{\g\ol\mu}+h_{\g\ol\b}h_{\a\ol\mu}).
\end{equation}
Moreover, the Bianchi identity $\nabla_\a R-\nabla^{\ol\b}R_{\a\ol\b}=i(n-1)\nabla^\b A_{\a\b}$ ({\cite[Lemma 2.2]{Lee}}) gives
\begin{equation}\label{A-div}
\nabla^\b A_{\a\b}=-\frac{2(n+1)i}{n}\nabla_\a P,
\end{equation}
and hence
\begin{equation}\label{V-pE}
V_{\a\ol\b\g}=\nabla_{\ol\b}A_{\a\g}+\frac{2i}{n}\Bigl((\nabla_\a P)h_{\g\ol\b}+(\nabla_\g P)h_{\a\ol\b}\Bigr).
\end{equation}
We also need the following identity for a pseudo-Einstein contact form on five dimensional CR manifolds:
\begin{equation}\label{A-0}
i(\nabla_0 A_{\a\b})A^{\a\b}= -\frac{1}{8}R\Delta_b R+|\nabla_{\ol\g}A_{\a\b}|^2+R_{\a\ol\b\g\ol\mu}A^{\a\g}A^{\ol\b\ol\mu}-\frac{1}{2}R|A_{\a\b}|^2+(\mbox{div}),
\end{equation}
where (div) denotes divergence terms and $\Delta_b=-\nabla_\a\nabla^\a-\nabla^\a\nabla_\a$ is the {\it sublaplacian}; see {\cite[Proposition 6.3]{HMM}} for the proof. Note that our sign convention of $\Delta_b$ is opposite to that in \cite{HMM}.

\subsection{CR weights and transformation formulas}\label{weights}
When we consider transformation formulas of pseudo-hermitian tensors under rescaling of the contact form, it is convenient to put appropriate CR weights on them. 

Let 
\[
K_M:=\{\zeta\in\mathbb{C}\otimes\wedge^{n+1}T^*M\ |\ \ol Z\lrcorner\,\zeta=0, \forall Z\in T^{1,0}M\}
\]
be the {\it CR canonical bundle}, and define the {\it CR density of weight} $(m, m), m\in\mathbb{Z}$ by
\[
\calE(m, m):=(K_M\otimes \ol{K}_M)^{-\frac{m}{n+2}}.
\]
We also call a section of this bundle a {\it CR density}. A choice of contact form $\theta$ determines a CR density $\tau_\theta:=|\zeta|^{-2/(n+2)}\in\calE(1, 1)$, where $\zeta\in K_M$ satisfies   
\[
\theta\wedge(d\theta)^n=i^{n^2}n!\theta\wedge(T\lrcorner\,\zeta)\wedge
(T\lrcorner\,\ol\zeta);
\]
such $\zeta$ is unique up to multiplications of $U(1)$-valued functions, so $\tau_\theta$ is well-defined. For a rescaling $\widehat\theta=e^\U\theta$, we have $\tau_{\widehat\theta}=e^{-\U}\tau_\theta.$ If a pseudo-hermitian tensor 
$B$ satisfies $\widehat B=e^{m\lambda}B$ for rescaling $\widehat\theta=e^{\lambda}\theta$ by any constant $\lambda\in\mathbb{R}$, then we put CR weight $(m, m)$ by replacing $B$ by $\tau_\theta^{\otimes m}B$. The resulting weighted tensor is invariant under rescaling of $\theta$ by constant functions. For example, as weighted tensors, the Levi form $h_{\a\ol\b}$ has CR weight $(1, 1)$ while the inverse $h^{\a\ol\b}$ has weight $(-1, -1)$. Thus, raising and lowering indices change the CR weight. The curvature tensor $R_{\a}{}^\b{}_{\g\ol\mu}$ and the torsion tensor $A_{\a\b}$ have weight $(0, 0)$, and the scalar curvature $R$ has weight $(-1, -1)$.  Putting these weights, we have the following transformation formulas for rescaling $\widehat\theta=e^\U\theta$:
\begin{align}
\widehat A_{\a\b}&=A_{\a\b}+i\U_{\a\b}-i\U_\a\U_\b, \label{A-trans} \\
\widehat P&=P+\frac{1}{2}\Delta_b\U-\frac{n}{2}\U_\a \U^\a, \label{P-trans}
\end{align}
where $\U_\a:=\nabla_\a \U$ and $\U_{\a\b}:=\nabla_\b\nabla_\a\U$. For a CR density $f\in\calE(m, m)$, the transformation formula of the sublaplacian is given by
\begin{equation}\label{sublap-trans}
\widehat\Delta_b f=\Delta_b f-(n+2m)(\U^\a\nabla_\a f+\U_\a\nabla^\a f)+m(\Delta_b\U)f
-2m(n+m)\U_\a \U^\a f.
\end{equation}

For a CR density $f\in\calE(-n-1, -n-1)$, we can define the integral 
\[
\int_M f:=\int_M (\tau_\theta^{n+1}f )\,\theta\wedge(d\theta)^n, 
\]
which is independent of choice of $\theta$. If the integral of a pseudo-hermitian invariant $F_\theta\in\calE(-n-1, -n-1)$ is independent of choice of a pseudo-Einstein contact form, it is called a {\it global secondary CR invariant}.

\subsection{The Graham--Lee connection} To compute explicit formulas of the Tanaka--Webster connection for the boundary of a strictly pseudoconvex domain, it is useful to introduce an extended connection, called the {\it Graham--Lee connection}, defined on a neighborhood of the boundary (\cite{GL}).

Let $X$ be an $(n+1)$-dimensional complex manifold, and $\Omega\subset X$ a bounded domain with strictly pseudoconvex boundary $M$. A real function $\rho\in C^\infty(X)$ is called a {\it (boundary) defining function} if $\Omega=\rho^{-1}((0, \infty))$ and $d\rho\neq0$ on $M$. The induced CR structure on $M$ is given by the subbundle ${\rm Ker}\, \partial\rho\subset \mathbb{C}\otimes TM$. The Levi form for the contact form $\theta=(i/2)(\pa\rho-\ol\pa\rho)|_{TM}$ is the restriction of $-\partial\ol\partial\rho$. 

We take a $(1, 0)$-vector field $\xi$ on a neighborhood of $M$ which satisfies 
\begin{equation}\label{xi-def}
\xi\rho=1, \quad -\xi\lrcorner\, \pa\ol\pa\rho=\kappa\ol\pa\rho
\end{equation}
with a real function $\kappa$ called the {\it transverse curvature}. Then we have the decomposition 
\[
T^{1, 0}X={\rm Ker}\, \pa\rho\oplus \mathbb{C}\xi
\]
near $M$. A local $(1, 0)$-frame $\{Z_\a, \xi\}$ adapted to this decomposition is called a {\it Graham--Lee frame}. In this frame, we can write as
\[
-\pa\ol\pa\rho=h_{\a\ol\b}\theta^\a\wedge\theta^{\ol\b}+\kappa\pa\rho\wedge\ol\pa\rho,
\]
where $\{\theta^\a, \pa\rho\}$ is the dual frame of $\{Z_\a, \xi \}$ and $h_{\a\ol\b}=-\pa\ol\pa\rho(Z_\a, Z_{\ol\b})$. Note that at the boundary, $h_{\a\ol\b}$ agrees with the Levi form for $\theta$. 

The Graham--Lee connection $\nabla$ for $\rho$ is a linear connection of $TX$ defined in a neighborhood of $M$. It preserves ${\rm Ker}\,\pa\rho$
 and satisfies $\nabla\xi=0$. The connection 1- forms $\varphi_\a{}^\b$ with respect to $\{Z_\a\}$ are characterized by the structure equations:
\begin{equation}\label{str-GL}
\begin{aligned}
d\theta^\a&=\theta^\b\wedge\varphi_\b{}^\a+iA^\a{}_{\ol\b}\pa\rho\wedge\theta^{\ol\b}-\kappa^\a\pa\rho\wedge\ol\pa\rho-\frac{1}{2}\kappa d\rho\wedge\theta^\a, \\
dh_{\a\ol\b}&=\varphi_\a{}^\g h_{\g\ol\b}+h_{\a\ol\g}\varphi_{\ol\b}{}^{\ol\g},
\end{aligned}
\end{equation}
where $\varphi_{\ol\b}{}^{\ol\g}:=\ol{\varphi_\b{}^\g}$ and the index in $\kappa^\a$ is raised by the inverse $h^{\a\ol\b}$ of $h_{\a\ol\b}$. We observe that these structure equations restrict to those of the Tanaka--Webster connection at the boundary.  Thus we have $\varphi_\a{}^\b|_{TM}=\omega_\a{}^\b$ and the restriction of $A^\a{}_{\ol\b}$ coincides with the Tanaka--Webster torsion tensor.

\section{Proof of Theorem \ref{main-thm}}
\subsection{Pseudo-hermitian invariants of weight $(-3,-3)$}
To prove Theorem \ref{main-thm}, we first list up all pseudo-hermitian invariants $F_\theta\in\calE(-3, -3)$ which may appear in the integrand of a global secondary CR invariant of CR five-manifolds:
\begin{prop}\label{list}
Let $F_\theta\in\calE(-3, -3)$ be a pseudo-hermitian invariant defined for pseudo-Einstein contact forms on five dimensional CR manifolds. Then it is a linear combination of 
\begin{equation*}
\begin{gathered}
P^3, \quad P|A_{\a\b}|^2, \quad P|S_{\a\ol\b\g\ol\mu}|^2, \quad P\Delta_b P, \quad |V_{\a\ol\b\g}|^2, \\
S_{\a\ol\b\g\ol\mu}A^{\a\g}A^{\ol\b\ol\mu}, \quad
S_\a{}^\b{}_\g{}^\mu S_\b{}^\nu{}_\mu{}^\tau S_\nu{}^\a{}_\tau{}^\g
\end{gathered}
\end{equation*}
modulo divergence terms.
\end{prop}
\begin{proof}
We first observe that we may assume that $F_\theta$ does not contain the index $0$ since we can eliminate the covariant derivative $\nabla_0$ by using the Ricci identities such as
\[
(\nabla_\a\nabla_{\ol\b}-\nabla_{\ol\b}\nabla_\a)B_\g{}^{\mu}
=-R_\g{}^\nu{}_{\a\ol\b}B_\nu{}^{\mu}+R_{\nu}{}^\mu{}_{\a\ol\b}B_\g{}^{\nu}-ih_{\a\ol\b}\nabla_0 B_\g{}^{\mu}
\]
for a pseudo-hermitian tensor $B_\g{}^{\mu}$.
Thus, $F_\theta$ is a linear combination of complete contractions of pseudo-hermitian tensors which are written in terms of 
\[
R_{\a}{}^\b{}_{\g\ol\mu}, \quad A_{\a\b}, \quad A_{\ol\a\ol\b}, \quad \nabla_\a, \quad \nabla_{\ol\a}.
\]
Note that these have CR weight $(0, 0)$. If we take trace with $h^{\a\ol\b}$ once, the CR weight changes by $(-1, -1)$. Since $F_\theta\in\calE(-3, -3)$, we must use $h^{\a\ol\b}$ exactly three times. Hence the possible combinations are the following:
\begin{align*}
&{\rm Case}\ 1.\quad A_{\a\b}, \quad A_{\ol\a\ol\b}, \quad R_\a{}^\b{}_{\g\ol\mu} \\
&{\rm Case}\ 2.\quad A_{\a\b}, \quad A_{\ol\a\ol\b}, \quad \nabla_\a, \quad \nabla_{\ol\a} \\
&{\rm Case}\ 3.\quad A_{\a\b}, \quad R_\a{}^\b{}_{\g\ol\mu}, \quad \nabla_{\ol\a}, \quad \nabla_{\ol\a} \quad {\rm and\ the\ complex\ conjugates}\\
&{\rm Case}\ 4.\quad A_{\a\b}, \quad \nabla_\a, \quad \nabla_{\ol\a}, \quad \nabla_{\ol\a}, \quad \nabla_{\ol\a}\quad {\rm and\ the\ complex\ conjugates}\\
&{\rm Case}\ 5.\quad R_\a{}^\b{}_{\g\ol\mu}, \quad R_\a{}^\b{}_{\g\ol\mu}, \quad R_\a{}^\b{}_{\g\ol\mu} \\
&{\rm Case}\ 6.\quad R_\a{}^\b{}_{\g\ol\mu}, \quad R_\a{}^\b{}_{\g\ol\mu}, \quad \nabla_\a, \quad \nabla_{\ol\a} \\
&{\rm Case}\ 7.\quad R_\a{}^\b{}_{\g\ol\mu}, \quad \nabla_\a, \quad \nabla_\a, \quad \nabla_{\ol\a}, \quad \nabla_{\ol\a}
\end{align*}
We can ignore Case 4 and Case 7 since they provide only divergence terms.
We will consider the other cases. In the sequel, we set
\[
\langle B, C, \cdots \rangle:={\rm span}_{\mathbb{C}}\{B, C, \cdots\}/({\rm divergence\ terms})
\]
for pseudo-hermitian invariants $B, C\cdots$. 
\bigskip

\underline{Case 1}.\ By \eqref{S-pE}, this case provides $S_{\a\ol\b\g\ol\mu}A^{\a\g}A^{\ol\b\ol\mu}$ and $P|A_{\a\b}|^2$.
\bigskip

\underline{Case 2}.\ Modulo divergence terms, we have the following two possibilities:
\begin{align}
\label{1-1} &{\rm contr}[(\nabla_{\ol\mu}A_{\a\b})(\nabla_\g A_{\ol\tau\ol\sigma})], \\
\label{1-2} &{\rm contr}[(\nabla_{\ol\mu}\nabla_\g A_{\a\b})A_{\ol\tau\ol\sigma}],
\end{align}
where ``contr'' means the complete contraction. We can express \eqref{1-1} in terms of $V_{\a\ol\b\g}, \nabla_\a P$ and their complex conjugates by using \eqref{V-pE}. Since $V_{\a\ol\b\g}$ is symmetric in $\a, \g$ and trace-free, we have
\[
{\rm contr}[(\nabla_{\ol\mu}A_{\a\b})(\nabla_\g A_{\ol\tau\ol\sigma})]\in
\langle |V_{\a\ol\b\g}|^2,\ |\nabla_\a P|^2\rangle
=\langle |V_{\a\ol\b\g}|^2,\ P\Delta_b P\rangle,
\]
where we have used $|\nabla_\a P|^2\equiv\frac{1}{2}P\Delta_b P$ modulo divergence. Thus, \eqref{1-1} appears in our list. By the Ricci identity, \eqref{1-2} can be rewritten as
\[
{\rm contr}[(\nabla_\g\nabla_{\ol\mu}A_{\a\b}+R_{\g\ol\mu}{}^\nu{}_\a A_{\nu\b}+R_{\g\ol\mu}{}^\nu{}_\b A_{\a\nu}+ih_{\g\ol\mu}\nabla_0 A_{\a\b})A_{\ol\tau\ol\sigma}].
\]
The first term is reduced to \eqref{1-1} modulo divergence while the second and the third term are reduced to Case 1. By \eqref{S-pE}, \eqref{V-pE} and \eqref{A-0}, the forth term gives
\[
i(\nabla_0 A_{\a\b})A^{\a\b}\in\langle P\Delta_b P,\ |V_{\a\ol\b\g}|^2,\ S_{\a\ol\b\g\ol\mu}A^{\a\g}A^{\ol\b\ol\mu},\ P|A_{\a\b}|^2 \rangle,
\]
which appears in the list.
\bigskip

\underline{Case 3}.\ We express $R_\a{}^\b{}_{\g\ol\mu}$ in terms of $S_\a{}^\b{}_{\g\ol\mu}$ and $P$. Modulo divergence, the term which involves $P$ gives $P\nabla^\a\nabla^\b A_{\a\b}$ and its complex conjugate, but this is a multiple of $P\Delta_b P$ modulo divergence by \eqref{A-div}. On the other hand, the terms which contain $S_\a{}^\b{}_{\g\ol\mu}$ are of the form
\[
{\rm contr}[(\nabla_{\ol\g}A_{\a\b})\nabla_{\ol\mu}S_{\nu\ol\tau\rho\ol\sigma}]
\]
or its complex conjugate modulo divergence. If the index $\ol\mu$ is contracted with $\a$ or $\b$, then it gives $0$ since $S_{\nu\ol\tau\rho\ol\sigma}$ must have an inner contraction. If the index $\ol\mu$ is contracted with $\nu$ or $\rho$, we obtain ${\rm contr}[(\nabla_{\ol\g}A_{\a\b})V_{\ol\tau \rho\ol\sigma}]$ by \eqref{div-S}, and this gives $|V_{\a\ol\b\g}|^2$ by \eqref{V-pE}.
\bigskip

\underline{Case 5}.\ We rewrite $R_\a{}^\b{}_{\g\ol\mu}$ by $S_\a{}^\b{}_{\g\ol\mu}$ and $P$. By the symmetries and trace-free property of $S_\a{}^\b{}_{\g\ol\mu}$, the resulting densities are $P^3$, $P|S_{\a\ol\b\g\ol\mu}|^2$, and 
\[
{\rm contr}[S_{\a_1\ol\b_1\g_1\ol\mu_1}S_{\a_2\ol\b_2\g_2\ol\mu_2}S_{\a_3\ol\b_3\g_3\ol\mu_3}].
\]
The last one gives 
\[
S_\a{}^\b{}_\g{}^\mu S_\b{}^\nu{}_\mu{}^\tau S_\nu{}^\a{}_\tau{}^\g, \quad 
S_\a{}^\b{}_\g{}^\mu S_\b{}^\nu{}_\tau{}^\g S_\nu{}^\a{}_\mu{}^\tau.
\]
However, on five dimensional CR manifolds, we have
\[
0=3!S_{[\a}{}^\b{}_{|\g|}{}^\mu S_\b{}^\nu{}_{|\mu|}{}^\tau S_{\nu]}{}^\a{}_\tau{}^\g
=S_\a{}^\b{}_\g{}^\mu S_\b{}^\nu{}_\mu{}^\tau S_\nu{}^\a{}_\tau{}^\g+
S_\nu{}^\b{}_\g{}^\mu S_\a{}^\nu{}_\mu{}^\tau S_\b{}^\a{}_\tau{}^\g,
\]
where $[\cdots]$ denotes the skew symmetrization over the indices $\a, \b, \nu$.
Hence we have only $S_\a{}^\b{}_\g{}^\mu S_\b{}^\nu{}_\mu{}^\tau S_\nu{}^\a{}_\tau{}^\g$ up to constant multiples.
\bigskip

\underline{Case 6}.\ We express $R_\a{}^\b{}_{\g\ol\mu}$ by $S_\a{}^\b{}_{\g\ol\mu}$ and $P$. A density made from $\{P, P, \nabla_\a, \nabla_{\ol\a} \}$ is only $P\Delta_b P$ modulo divergence. A density made from $\{P, S_{\a\ol\b\g\ol\mu}, \nabla_\a, \nabla_{\ol\a} \}$ is 0 since we must take trace of $S_{\a\ol\b\g\ol\mu}$. We will consider the densities made from $\{S_{\a\ol\b\g\ol\mu}, S_{\a\ol\b\g\ol\mu}, \nabla_\a, \nabla_{\ol\a}\}$. Up to divergence terms and the complex conjugates, these are of the form
\[
{\rm contr}[(\nabla_{\tau}\nabla_{\ol\sigma}S_{\a_1\ol\b_1\g_1\ol\mu_1})S_{\a_2\ol\b_2\g_2\ol\mu_2}].
\]
We devide the cases according to the index with which $\ol\sigma$ is contracted:
\begin{align}
\label{6-1} 
&{\rm contr}[(\nabla_{\tau}\nabla^{\a_1}S_{\a_1\ol\b_1\g_1\ol\mu_1})S_{\a_2\ol\b_2\g_2\ol\mu_2}],  \\
\label{6-2}
&{\rm contr}[(\nabla_{\tau}\nabla^{\a_2}S_{\a_1\ol\b_1\g_1\ol\mu_1})S_{\a_2\ol\b_2\g_2\ol\mu_2}], \\
\label{6-3}
&{\rm contr}[(\nabla_{\tau}\nabla^{\tau}S_{\a_1\ol\b_1\g_1\ol\mu_1})S_{\a_2\ol\b_2\g_2\ol\mu_2}].
\end{align}

By \eqref{div-S}, \eqref{6-1} gives 
\begin{equation}\label{nabla-V}
{\rm contr}[(\nabla_{\tau}
V_{\ol\b_1\g_1\ol\mu_1})S_{\a_2\ol\b_2\g_2\ol\mu_2}]\equiv
-{\rm contr}[V_{\ol\b_1\g_1\ol\mu_1}\nabla_{\tau}S_{\a_2\ol\b_2\g_2\ol\mu_2}].
\end{equation}
In the right-hand side, the index $\tau$ must be contracted with $\ol\beta_2$ or $\ol\mu_2$ since otherwise $S_{\a_2\ol\b_2\g_2\ol\mu_2}$ has an inner contraction. Hence this gives $|V_{\a\ol\b\g}|^2$.

Next we consider \eqref{6-2}. If we contract $\tau$ with $\ol\beta_2$ or $\ol\mu_2$, \eqref{6-2} vanishes since $S_{\a_1\ol\b_1\g_1\ol\mu_1}$ has an inner contraction. Hence $\tau$ is contracted with $\ol\beta_1$ or $\ol\mu_1$, and \eqref{6-2} gives only
\[
(\nabla_{\tau}\nabla^{\a_2}S_{\a_1}{}^\tau{}_{\g}{}^{\mu})S_{\a_2}{}^{\a_1}{}_{\mu}{}^{\g}.
\]
By the Ricci identity, this is computed as
\[
(\nabla^{\a_2}\nabla_{\tau}S_{\a_1}{}^\tau{}_{\g}{}^{\mu}
+R\# S-i\d_\tau{}^{\a_2}\nabla_0S_{\a_1}{}^\tau{}_{\g}{}^{\mu})S_{\a_2}{}^{\a_1}{}_{\mu}{}^{\g},
\]
where $R\# S$ denotes the terms obtained by the curvature action. The first term is the complex conjugate of \eqref{6-1}, and the second term is reduced to Case 5. The third term gives $\frac{-i}{2}\nabla_0 |S_{\a\ol\b\g\ol\mu}|^2=\frac{1}{4}(\nabla_\a\nabla^\a-\nabla^\a\nabla_\a)|S_{\a\ol\b\g\ol\mu}|^2$, which is a divergence.

We will consider \eqref{6-3}. The complex conjugate gives 
\[
(\nabla^\tau\nabla_\tau S_{\a\ol\b\g\ol\mu})S^{\ol\b\a\ol\mu\g}.
\]
By \eqref{V-S}, this is computed as
\[
\bigl(\nabla^\tau\nabla_\a S_{\tau\ol\b\g\ol\mu}+2i\nabla^\tau (V_{\g\ol\b[\tau}h_{\a]\ol\mu}+V_{\g\ol\mu[\tau}h_{\a]\ol\b})\bigr)S^{\ol\b\a\ol\mu\g}.
\]
The first term is reduced to the complex conjugate of \eqref{6-2}. The second and the third terms coincide with the complex conjugate of \eqref{nabla-V} and yield $|V_{\a\ol\b\g}|^2$.

Thus we complete the proof.
\end{proof}

Given a pseudo-hermitian invariant $F_\theta\in\calE(-3, -3)$, we can eliminate the terms $P^3$ and $|V_{\a\ol\b\g}|^2$ in $F_\theta$ by adding multiples of $Q'$ and $\mathcal{I}'$; see the formulas \eqref{Q-prime}, \eqref{I-prime}.
Hence, to prove Theorem \ref{main-thm}, it suffices to show that if an integral
\[
\int_M (c_1 P|A_{\a\b}|^2+c_2 P\Delta_b P+c_3 S_{\a\ol\b\g\ol\mu}A^{\a\g}A^{\ol\b\ol\mu}+c_4 P|S_{\a\ol\b\g\ol\mu}|^2)
\]
is a global secondary CR invariant of CR five-manifolds then $c_1=c_2=c_3=c_4=0$. We prove this fact by considering two examples of CR manifold.
\subsection{The CR sphere} First we deal with the CR sphere 
\[
S^5=\Bigl\{(z^\a, w)\in\mathbb{C}^3\ \Big|\ \sum_{\a=1}^2 |z^\a|^2+|w|^2=1\Bigr\},
\]
on which $S_{\a\ol\b\g\ol\mu}=0$. We will prove that if
\begin{equation}\label{inv-sphere}
\int_{S^5} (c_1 P|A_{\a\b}|^2+c_2 P\Delta_b P)
\end{equation}
is independent of choice of pseudo-Einstein contact form, then $c_1=c_2=0$.

We first compute the Graham--Lee connection for the defining function $\rho:=1-\sum|z^\a|^2-|w|^2$. The $(1, 0)$-vector field
\[
\xi=\frac{1}{\rho-1}\Bigl(z^\a\frac{\pa}{\pa z^\a}+w\frac{\pa}{\pa w}\Bigr)
\]
satisfies \eqref{xi-def} with the transverse curvature $\kappa=(1-\rho)^{-1}$.
We take a Graham--Lee frame $\{Z_\a, \xi\}$ and its dual coframe $\{\theta^\a, \pa\rho\}$ as
\[
Z_\a=\frac{\pa}{\pa z^\a}-\frac{z^{\ol\a}}{\ol w}\frac{\pa}{\pa w}, \quad 
\theta^\a=dz^\a+\frac{z^\a}{1-\rho}\pa\rho.
\]
Then, the Levi form $h_{\a\ol\b}=-\pa\ol\pa\rho(Z_\a, Z_{\ol\b})$ and its inverse are given by
\[
h_{\a\ol\b}=\d_{\a\ol\b}+\frac{z^{\ol\a}z^\b}{|w|^2}, \quad 
h^{\a\ol\b}=\d^{\a\ol\b}-\frac{z^\a z^{\ol\b}}{1-\rho}.
\]
If we set
\[
\varphi_\b{}^\a=\frac{1}{1-\r}z^\a h_{\b\ol\g}\th^{\ol\g}+\frac{1}{2(1-\rho)}(\pa\rho-\ol\pa\rho)\d_\b{}^\a, \quad A^\a{}_{\ol\b}=0,
\]
then these satisfy the structure equations \eqref{str-GL}. Hence the Tanaka--Webster connection and the torsion tensor for the standard contact form $\theta=(i/2)(\pa\rho-\ol\pa\rho)|_{TS^5}$ are
\[
\omega_\b{}^\a=\varphi_\b{}^\a|_{TS^5}=z^\a h_{\b\ol\g}\th^{\ol\g}-i\theta \d_\b{}^\a, \quad A_{\a\b}=0.
\]
The Tanaka--Webster curvature tensor is given by
\[
R_{\a\ol\b\g\ol\mu}=h_{\a\ol\b}h_{\g\ol\mu}+h_{\g\ol\b}h_{\a\ol\mu}.
\]
In particular, we have $P=1$. In the sequel, we trivialize the CR density by $\theta$.

Let $\widehat\theta=e^\U\theta$ be another pseudo-Einstein contact form. Then, by \eqref{P-trans} and \eqref{sublap-trans}, we have
\begin{equation}\label{Delta-P-trans}
\widehat\Delta_b \widehat P=\frac{1}{2}\Delta_b^2 \U-\Delta_b \U-\Delta_b(\U_\a \U^\a)-\frac{1}{2}(\Delta_b \U)^2+2\U_\a\U^\a+O(3),
\end{equation}
where $O(3)$ denotes terms of order greater than or equal to 3 in $\U$. We consider a one parameter family of pseudo-Einstein contact forms $e^{\varepsilon\U}\theta\ (\varepsilon\in\mathbb{R})$ and set
\[
\d^k:=\frac{\pa^k}{\pa \varepsilon^k}\Big|_{\varepsilon=0}.
\]
By \eqref{A-trans}, \eqref{P-trans}, and \eqref{Delta-P-trans}, we have
\begin{gather*}
\d A_{\a\b}=i\U_{\a\b}, \quad \d P=\frac{1}{2}\Delta_b \U, \quad 
\d (\Delta_b P)=\frac{1}{2}\Delta_b^2 \U-\Delta_b \U, \\ 
\frac{1}{2}\d^2 (\Delta_b P)\equiv -\frac{1}{2}(\Delta_b \U)^2+2\U_\a\U^\a\ {\rm mod\ divergence},
\end{gather*}
and hence
\begin{align*}
\frac{1}{2}\d^2(P|A_{\a\b}|^2)&=(\d A_{\a\b})( \d A^{\a\b})=\U_{\a\b}\U^{\a\b}, \\
\frac{1}{2}\d^2(P\Delta_b P)&=\d P\d(\Delta_b P)+\frac{1}{2}\d^2(\Delta_b P)\equiv \frac{1}{4}\Delta_b \U \Delta_b^2 \U-(\Delta_b \U)^2+2\U_\a\U^\a
\end{align*}
modulo divergence. Thus, if the integral \eqref{inv-sphere} is a global secondary CR invariant of $S^5$, we have
\begin{equation}\label{int-U-sphere}
\int_{S^5} \Bigl[c_1 \U_{\a\b}\U^{\a\b}+c_2\Bigl(\frac{1}{4}\Delta_b \U \Delta_b^2 \U-(\Delta_b \U)^2+2\U_\a\U^\a\Bigr)\Bigr]\theta\wedge(d\theta)^2=0
\end{equation}
for any CR pluriharmonic function $\U$.

Let us consider the CR pluriharmonic functions
\[
\U:=\frac{1}{m}\bigl((z^1)^m+(z^{\ol1})^m\bigr)\big|_{S^5}, \quad m=1, 2, \cdots.
\]
For simplicity, we do not refer to the restriction to $S^5$ in the sequel. Since $\omega_\a{}^\b(Z_\g)=0$ in the frame $\{Z_\a\}$, we have
\[
\U_1 =(z^1)^{m-1}, \quad \U_{11}=Z_1 Z_1\U=(m-1)(z^1)^{m-2}
\]
and the other components in $\U_\a, \U_{\a\b}$ are $0$. Thus, 
\begin{align*}
\U_\a\U^\a&=h^{1\ol1}\U_1\U_{\ol1}=(1-|z^1|^2)|z^1|^{2(m-1)}, \\
\U_{\a\b}\U^{\a\b}&=h^{1\ol1}h^{1\ol1}\U_{11}\U_{\ol1\ol1}=(m-1)^2(1-|z^1|^2)^2|z^1|^{2(m-2)}.
\end{align*}
From
\[
\U_{\a\ol\b}=Z_{\ol\b}Z_\a\U-\omega_\a{}^\g(Z_{\ol\b})\U_\g=-(z^1)^m h_{\a\ol\b},
\]
we have
\[
\Delta_b\U=2m\U, \quad \Delta_b^2\U=4m^2\U.
\]
It follows that
\begin{align*}
&\frac{1}{4}\Delta_b \U \Delta_b^2 \U-(\Delta_b \U)^2+2\U_\a\U^\a \\
&\quad =2(2m-5)|z^1|^{2m}+2|z^1|^{2(m-1)}+2(m-2)\bigl((z^1)^{2m}+(z^{\ol1})^{2m}\bigr).
\end{align*}
In the integration \eqref{int-U-sphere}, we may replace $\theta\wedge(d\theta)^2$ by the standard volume form $dV$ since the former is a multiple of the latter. We first note that 
\[
\int_{S^5} (z^1)^{2m}dV=0
\]
holds since the integral changes sign by the rotation $z^1\mapsto e^{i\pi/(2m)}z^1$. We will compute the integral of $|z^1|^{2k}$ by using the polar coordinates of $S^5$:
\begin{align*}
&\cos\theta_1 \\
&\sin\theta_1\cos\theta_2 \\
&\sin\theta_1\sin\theta_2\cos\theta_3 \\
&\sin\theta_1\sin\theta_2\sin\theta_3 \cos\theta_4 \\
&\sin\theta_1\sin\theta_2\sin\theta_3\sin\theta_4\cos\theta_5 \\
&\sin\theta_1\sin\theta_2\sin\theta_3\sin\theta_4\sin\theta_5, 
\quad 0\le\theta_1, \theta_2, \theta_3, \theta_4\le \pi, \ 0\le\theta_5\le2\pi.
\end{align*}
We let ${\rm Re}\, z^1, {\rm Im}\,z^1$ be the last two coordinates. Since the volume is given by
\[
dV=\sin^4\theta_1\sin^3\theta_2\sin^2\theta_3\sin\theta_4 d\theta_1d\theta_2d\theta_3d\theta_4d\theta_5,
\]
we have
\begin{align*}
\int_{S^5} |z^1|^{2k} dV&=\int \sin^{2k+4}\theta_1\sin^{2k+3}\theta_2\sin^{2k+2}\theta_3\sin^{2k+1}\theta_4 d\theta_1d\theta_2d\theta_3d\theta_4d\theta_5 \\
&=\frac{2\pi^3}{(k+1)(k+2)}.
\end{align*}
Here we have used the formula
\begin{equation*}
\int_0^{\frac{\pi}{2}} \sin^n\theta d\theta=
\begin{cases}
\frac{(n-1)!!}{n!!} & (n: {\rm odd}) \\
\frac{(n-1)!!}{n!!}\cdot \frac{\pi}{2} & (n: {\rm even}).
\end{cases}
\end{equation*}
By computing with these formulas, we obtain
\begin{align*}
\frac{1}{2\pi^3}\int_{S^5} \U_{\a\b}\U^{\a\b}dV&=\frac{6(m-1)}{m(m+1)(m+2)}, \\
\frac{1}{2\pi^3}\int_{S^5}\Bigl(\frac{1}{4}\Delta_b \U \Delta_b^2 \U-(\Delta_b \U)^2+2\U_\a\U^\a\Bigr)dV&=\frac{4(m-1)^2}{m(m+1)(m+2)}.
\end{align*}
It follows from \eqref{int-U-sphere} that 
\[
3c_1+2(m-1)c_2=0
\]
 for $m=2, 3, \cdots$ and hence $c_1=c_2=0$.
\subsection{A Reinhardt hypersurface}\label{reinhardt} We next consider the boundary of the Reinhardt domain
\[
\Omega:=\Bigl\{(\zeta^0, \zeta^1, \zeta^2)\in\mathbb{C}^3\ \Big|\ \sum_{i=0}^2(\log|\zeta^i|)^2<1\Bigr\}.
\]
Pseudo-hermitian geometry of $\pa\Omega$ is examined in \cite{Mar1}, and higher dimensional cases are discussed in \cite{CT}. By the coordinate change $z^i=\log\zeta^i$, the boundary of $\Omega$ is mapped to
\[
M=\Bigl\{(z^0, z^1, z^2)\in\mathbb{C}^3\ \Big|\ \sum_{i=0}^2(x^i)^2=1\Bigr\}\Big/(2\pi i\mathbb{Z})^3,
\]
where $z^j=x^j+iy^j$. We set $w=z^0, (z^\a)=(z^1, z^2)$. We will show that if 
\begin{equation*}
\int_M (c_3 S_{\a\ol\b\g\ol\mu}A^{\a\g}A^{\ol\b\ol\mu}+c_4 P|S_{\a\ol\b\g\ol\mu}|^2)
\end{equation*}
is a global secondary CR invariant of $M$, then $c_3=c_4=0$.

We take the defining function 
\[
\rho:=2\bigl(1-(x^0)^2-\textstyle \sum(x^\a)^2\bigr)
\]
and trivialize CR densities by the associated contact form $\theta=(i/2)(\pa\rho-\ol\pa\rho)|_{TM}$. We will compute the Graham--Lee connection for $\rho$. We define a $(1, 0)$-vector field $\xi$ by
\[
\xi=\frac{x^i}{\rho-2}\frac{\pa}{\pa z^i}.
\]
Then $\xi$ satisfies \eqref{xi-def} with the transverse curvature $\kappa=(2(2-\rho))^{-1}$. We take the Graham--Lee frame $\{Z_\a, \xi\}$ and the coframe $\{\theta^\a, \pa\rho\}$ defined by 
\[
Z_\a:=\frac{\pa}{\pa z^\a}-\frac{x^\a}{x^0}\frac{\pa}{\pa w}, \quad 
\theta^\a:=dz^\a+\frac{x^\a}{2-\rho}\pa\rho.
\]
Then the Levi form $h_{\a\ol\b}=-\pa\ol\pa\rho(Z_\a, Z_{\ol\b})$ and its inverse are given by
\[
h_{\a\ol\b}=\d_{\a\ol\b}+\frac{x^\a x^\b}{(x^0)^2}, \quad
h^{\a\ol\b}=\d^{\a\ol\b}-\frac{2x^\a x^\b}{2-\rho}.
\]
If we set
\[
\varphi_\b{}^\a:=\frac{1}{2-\rho}x^\a h_{\b\ol\g}(\th^\g+\th^{\ol\g})+\frac{1}{4(2-\rho)}(\pa\rho-\ol\pa\rho)\d_\b{}^\a, \quad 
A^\a{}_{\ol\b}:=\frac{i}{2(2-\rho)}\d_{\ol\b}{}^\a,
\]
then these satisfy the structure equations \eqref{str-GL}. Hence the Tanaka--Webster connection and the torsion tensor for $\theta$ are given by
\[
\omega_\b{}^\a=\varphi_\b{}^\a|_{TM}=\frac{1}{2}x^\a h_{\b\ol\g}(\th^\g+\th^{\ol\g})-\frac{i}{4}\th\d_\b{}^\a, \quad A_{\a\b}=-\frac{i}{4}h_{\a\ol\b}.
\]
The Tanaka--Webster curvature tensor satisfies
\[
R_\b{}^\a{}_{\g\ol\mu}=\frac{1}{4}(\d_\b{}^\a h_{\g\ol\mu}+\d_\g{}^\a h_{\b\ol\mu}-16A^\a{}_{\ol\mu}A_{\g\b}).
\]
It follows that $\theta$ is pseudo-Einstein and 
\begin{equation*}
S_{\a\ol\b\g\ol\mu}=\frac{1}{12}(h_{\a\ol\b}h_{\g\ol\mu}+h_{\g\ol\b}h_{\a\ol\mu})-4A_{\a\g}A_{\ol\b\ol\mu}.
\end{equation*}
We remark that from this formula one can check that the local CR invariant 
\[
S_\a{}^\b{}_\g{}^\mu S_\b{}^\nu{}_\mu{}^\tau S_\nu{}^\a{}_\tau{}^\g
\]
in Theorem \ref{main-thm} is a nonzero constant function on $M$. 

Let us consider the one parameter family of pseudo-Einstein contact forms
$e^{\varepsilon\U}\theta$, where 
\[
\U:=2x^1=z^1+z^{\ol1}
\]
is a well-defined CR pluriharmonic function on $M$. Since $\U_\a=\d_\a{}^1$, we have
\[
\U_\a\U^\a=h^{1\ol1}=1-(x^1)^2, \quad A_{\a\b}\U^\a\U^\b=h^{\a\ol1}h^{\b\ol1}A_{\a\b}=-\frac{i}{4}\bigl(1-(x^1)^2\bigr).
\]
By using these formulas and \eqref{A-trans}, we compute
\begin{align*}
\frac{1}{24}\d^4(S_{\a\ol\b\g\ol\mu}A^{\a\g}A^{\ol\b\ol\mu})
&=\frac{1}{4}S_{\a\ol\b\g\ol\mu}\d^2(A^{\a\g})\d^2(A^{\ol\b\ol\mu}) \\
&=S_{\a\ol\b\g\ol\mu}\U^\a\U^\g\U^{\ol\b}\U^{\ol\mu} \\
&=\frac{1}{6}(\U_\a\U^\a)^2-4|A_{\a\b}\U^\a\U^\b|^2 \\
&=-\frac{1}{12}\bigl(1-(x^1)^2\bigr)^2.
\end{align*}
On the other hand, $\d^4(P|S_{\a\ol\b\g\ol\mu}|^2)=0$ by \eqref{P-trans}. Thus we have
\[
c_3\int_M  \bigl(1-(x^1)^2\bigr)^2 \theta\wedge(d\theta)^2=0,
\]
which implies $c_3=0$. It then follows from
\[
\frac{1}{2}\d^2(P|S_{\a\ol\b\g\ol\mu}|^2)=-|S_{\a\ol\b\g\ol\mu}|^2\U_\a\U^\a=
-\frac{1}{6}\bigl(1-(x^1)^2\bigr)
\]
that
\[
c_4\int_M  \bigl(1-(x^1)^2\bigr) \theta\wedge(d\theta)^2=0,
\]
which implies $c_4=0$.

Thus we complete the proof of Theorem \ref{main-thm}.


\end{document}